\documentclass[12pt,reqno]{amsart}
\usepackage{lipsum}
\usepackage{a4wide}
\usepackage{amssymb} 
\usepackage{amsmath}
\usepackage{amsthm} 
\usepackage{tikz-cd} 
\usepackage{amscd}
\usepackage{comment}
\usepackage{enumitem}
\usepackage{bbold}  
\theoremstyle{plain}
\usepackage[margin=2.4cm]{geometry}
\usepackage[strings]{underscore}
\usepackage{url}
\usepackage{hyperref}
\numberwithin{equation}{section}

\newcommand{\ind}{\operatorname{Ind}}

\newcommand{\res}{\operatorname{Res}}

\newtheorem{theorem}{Theorem}[section]
\newtheorem{corollary}[theorem]{Corollary} 
\newtheorem{lemma}[theorem]{Lemma}
\newtheorem{remark}[theorem]{Remark}
\newtheorem{proposition}[theorem]{Proposition}

\title{Twisted Jacquet modules: a conjecture of D. Prasad}
\author{Santosh Nadimpalli}
\address{Department of Mathematics and Statistics, Indian Institute of Technology Kanpur, Kanpur - 208016, India} 
\email{nsantosh@iitk.ac.in}
\author{Mihir Sheth}
\address{Department of Mathematics, Indian Institute of Science, Bangalore - 560012, India}
\email{mihirsheth@iisc.ac.in}
\begin{document}

\begin{abstract}

In this note, we study the twisted Jacquet modules of sub-quotients of
principal series representations of ${\rm GL}_2(D)$ where $D$ is a
division algebra over a non-archimedean local field $F$. We begin with a proof of a conjecture of D. Prasad on twisted Jacquet modules of Speh representations of ${\rm GL}_2(D)$ when $D$
is the quaternionic division algebra. For arbitrary division algebras $D$ over $F$, we focus on depth-zero principal series. We compute the dimensions of twisted Jacquet modules of
generalized Speh representations and explicitly investigate their structure as $D^{\times}$-representations in the depth-zero situation.  
\end{abstract}

\maketitle
\section{Introduction}

The multiplicity one property of the space of Whittaker models has been a central result in the representation theory of quasi-split reductive groups over local fields. The dimension of the space of generalized Whittaker models is a useful invariant to measure the growth of a representation and can be greater than one for general $p$-adic reductive groups. It is also intricately related to some well-known branching problems. Nonetheless, the spaces of Whittaker models seem to be far from being well understood, especially for non-quasi-split $p$-adic groups. In this note, we consider the group ${\rm GL}_2(D)$ where $D$ is a division algebra over a non-archimedean local field $F$ and study non-degenerate Whittaker models, also known as the twisted Jacquet modules, of smooth irreducible representations of ${\rm GL}_2(D)$. The twisted Jacquet modules for supercuspidal representations of ${\rm GL}_2(D)$ were first studied by Raghuram and Prasad (see \cite[Propostion 1]{dipendragermexpansions}). Here, we focus on irreducible sub-quotients of principal series representations and describe the structure of their twisted Jacquet modules explicitly in the depth-zero situation. 


To fix some notations, let $\tau$ be an
irreducible smooth representation of $D^\times$. Let $\nu_\tau$ be an
unramified character of $D^\times$ such that the normalised induction
$\tau\nu_\tau^{-1/2}\times \tau\nu_\tau^{1/2}$ is reducible and the
generalised Steinberg representation ${\rm St}(\tau)$ occurs as the
quotient. The irreducible sub-representation of
$\tau\nu_\tau^{-1/2}\times \tau\nu_\tau^{1/2}$, denoted by
${\rm Sp}(\tau)$, is the Speh representation associated with $\tau$. Let $B$ be the minimal parabolic subgroup of ${\rm GL}_2(D)$ consisting of upper triangular matrices with
unipotent radical $N$. Let $\psi:F\rightarrow \mathbb{C}^\times$ be a non-trivial additive character on $F$, viewed as the character on $N$ via
\[ \psi\left(\begin{pmatrix}1&x\\0&1\end{pmatrix}\right)=\psi({\rm Tr}_{D/F}(x)),\]
where ${\rm Tr}_{D/F}$ is the reduced trace. The twisted Jacquet
module of a smooth representation $(\pi,V)$ of ${\rm GL}_2(D)$ is the space $\pi_{N, \psi}$ of $\psi$-coinvariants of $N$ in $V$, and is naturally a representation of $D^\times$.

Based on a multiplicity one result of Rallis, D. Prasad made a conjecture in \cite{pr01} for the quaternionic division algebra $D$ which predicts that ${\rm Sp}(\tau)_{N, \psi}$ is a character and precisely describes this character as a $D^{\times}$-representation. We first prove this conjecture (see Theorem \ref{dpthm}): 
\begin{theorem}\label{introthm1}
	Let $D$ be the quaternionic division algebra and $\tau$ be a smooth irreducible representation of $D^\times$ of dimension $>1$. Then the $D^\times$-representation ${\rm Sp}(\tau)_{N, \psi}$ is isomorphic to $\omega_\tau\circ {\rm Nr}_{D/F}$ where $\omega_\tau$ is the central character of $\tau$ and ${\rm
		Nr}_{D/F}$ is the reduced norm map of $D$.  
\end{theorem}
The above theorem is proved by comparing the germ expansions of representations of ${\rm GL}_2(D)$ and ${\rm GL}_{4}(F)$ which correspond to each other under the Jacquet--Langlands correspondence. This gives
that the twisted Jacquet module ${\rm Sp}(\tau)_{N, \psi}$ is one-dimensional. The explicit action of $D^\times$ on this one-dimensional space is given by  a result of Gan and Takeda on Shalika models of
${\rm GL}_2(D)$. The idea of comparing local character expansions of representations to study twisted Jacquet modules for non-quasi-split groups has been recently used by Y. Cai to construct a family of Speh representations having unique models of degenerate type \cite{cai23}.

For an arbitrary division algebra $D$, we focus on depth-zero principal series of ${\rm GL}_2(D)$. Let $I$ be the standard Iwahori subgroup and
$I(1)$ be the pro-$p$ radical of $I$. Assume that $\psi$
restricts to a non-trivial character $\psi_0$ on $I(1)\cap N$. For any smooth representation $(\sigma, W)$ of $I$, let
$W^{I(1), \psi_0}$ denote the space $\{ w\in W:
\sigma(g)w=\psi_0(g)w, \forall g\in I(1)\}$. We prove the following theorem which can be considered as the analogue of the result of Moy and Prasad on the compatibility of $I(1)$-invariants with Jacquet modules (see Theorem \ref{main}).
\begin{theorem}
  Let $\tau_1$ and $\tau_2$ be two irreducible depth-zero
  representations of $D^\times$. Then the natural
  map
  \[(\tau_1\times \tau_2)^{I(1), \psi_0}\rightarrow (\tau_1\times
    \tau_2)_{N, \psi}\] is an isomorphism.
\end{theorem}
Using the above theorem, we prove that the natural maps
\[{\rm Sp}(\tau)^{I(1), \psi_0}\rightarrow {\rm Sp}(\tau)_{N, \psi}\]
and 
\[{\rm St}(\tau)^{I(1), \psi_0}\rightarrow {\rm St}(\tau)_{N, \psi}\]
are isomorphisms. The results of Minguez and Secherre in \cite[\S 5.3]{ms14} on the functor
$\mathbf{K}$ describe a crucial part of the space of invariants for the first principal congruence subgroup $K(1)$. Putting together the
results of Minguez and Secherre with the above isomorphisms, we obtain for depth zero $\tau$ that
the dimension of ${\rm Sp}(\tau)_{N, \psi}$ is equal to $d(d-1)/2$
where $d$ is the dimension of $\tau$ (see Corollary \ref{dimformulae}). Note that Speh representations {\it no} longer support unique Whittaker models when $d>2$. In the case
where $d$ is odd, we show that the $D^\times$-representation ${\rm Sp}(\tau)_{N, \psi}$
is isomorphic to the exterior square representation (see Proposition \ref{oddcaseprop}). The case where $d=2$ is arithmetically more involved and we use some computations on Gauss sums to determine the explicit structure of the twisted Jacquet
module of ${\rm Sp}(\tau)$. In this context, we generalize Theorem \ref{introthm1} to arbitrary division algebras and obtain a different proof of it in the situation of the quaternionic division algebra (see Theorem
\ref{explicitstructure}). 
\\

\noindent {\it Acknowledgements}: The authors would like to thank Dipendra Prasad for his valuable
suggestions and the proof of Theorem 3.1. We also have benefited from the discussions with Guy Henniart and C. S. Rajan and we thank them for their
suggestions. The authors thank the anonymous referee for several useful comments on an earlier version of the paper. The first author thanks DST-INSPIRE for the research grant.
\section{Preliminaries}
We fix some notation and recall some facts. 
\subsection{}
Let $F$ be a non-archimedean local field of residue characteristic $p$, $\mathfrak{o}_F$ be
the ring of integers in $F$, $\mathfrak{p}_F\subseteq\mathfrak{o}_F$ be the maximal
ideal, and $\mathbb{F}_q$ be the residue field
of $F$ of cardinality $q$. Let $D$ be a central division algebra over $F$ index $n$. The maximal
order of $D$ is denoted by $\mathfrak{o}_D$ and the maximal ideal of
$\mathfrak{o}_D$ is denoted by $\mathfrak{p}_D$. For a central
simple algebra $A$ over a field $k$, the reduced norm map (resp. the reduced trace map) is
denoted by ${\rm Nr}_{A/k}$ (resp.  ${\rm Tr}_{A/k}$). Similarly, for a finite field extension $l/k$, the field norm map (resp. the field trace map) is
denoted by ${\rm Nr}_{l/k}$ (resp.  ${\rm Tr}_{l/k}$). Let $\varpi_{F}$ and $\varpi_{D}$ be the uniformizers of $F$ and $D$ respectively such that $\varpi_F=\varpi_{D}^{n}$. Then ${\rm Nr}_{D/F}(\varpi_{D})=(-1)^{n+1}\varpi_{F}$. Let $|\cdot|_{F}$ and $|\cdot|$ denote the normalized non-archimedean absolute values on $F$ and $D$ respectively such that $|\varpi_{F}|_{F}=q^{-1}$ and $|\varpi_{D}|=q^{-n}$. For
$z\in \mathbb{C}^\times$, we denote by $\mu_z$ the unramified character
of $D^\times$ which sends  $\varpi_{D}$ to $z$. 
\subsection{}
For a divisor $d$ of $n$  with $n=md$, let $F_{d}$ denote the unramified extension of $F$ of degree $d$ viewed as a subfield of $D$, and $D_{m}$ denote the centralizer of $F_{d}$ in $D$. The algebra $D_{m}$ is a central division algebra over $F_{d}$ of index $m$. Let $\theta:F_{d}^{\times}\rightarrow \mathbb{C}^{\times}$ be a tamely ramified character all whose Galois conjugates are distinct. Composing it with the reduced norm $\mathrm{Nr}_{D_{m}/F_{d}}:D_{m}^{\times}\rightarrow F_{d}^{\times}$ and extending it to $D^{\times}_{m}D(1)$ by declaring it to be trivial on $D(1)=1+\varpi_{D}\mathcal{O}_{D}$, we have a character $\tilde{\theta}:D^{\times}_{m}D(1)\rightarrow \mathbb{C}^{\times}$. Note that  $D^{\times}_{m}D(1)=\mathfrak{o}_D^{\times}\rtimes\varpi_{D}^{d\mathbb{Z}}$. Inducing $\tilde{\theta}$ to $D^{\times}$, we obtain a smooth tamely ramified irreducible $d$-dimensional representation $\mathrm{Ind}_{D^{\times}_{m}D(1)}^{D^{\times}}\tilde{\theta}$ of $D^{\times}$. All smooth tamely ramified irreducible representations of $D^{\times}$ are obtained in this fashion \cite{sz05}.
\subsection{}
Let $G$ be the group ${\rm GL}_2(D)$. Let $B\subseteq G$ be the subgroup of upper triangular matrices (the standard minimal parabolic subgroup), $N\subseteq B$ be the subgroup of upper triangular unipotent matrices (the unipotent radical of $B$), and $T\subseteq B$ be the subgroup of diagonal matrices (the Levi quotient of $B$). 
The group $D^\times$ is viewed as a subgroup of $T$ sitting diagonally in it.
We denote by $K$ the maximal compact subgroup ${\rm GL}_2(\mathfrak{o}_D)$ of $G$. Let $I$
  denote the standard Iwahori subgroup of $G$ and $K(1)$ and $I(1)$ be the
  pro-$p$ radicals of $K$ and $I$ respectively. Let $T_{0}=T\cap K$ and $s=\left(
  \begin{smallmatrix}0&1\\-1&0\end{smallmatrix}\right)$. We denote the subgroup of upper triangular (resp. unipotent) matrices of $\mathrm{GL}_{2}(\mathbb{F}_{q^{n}})$ by $B(\mathbb{F}_{q^{n}})$ (resp. $U(\mathbb{F}_{q^{n}})$). A non-trivial additive
  (smooth) character $\psi_{F}:F \rightarrow \mathbb{C}^\times$ gives
  rise to a non-trivial additive character
  $\psi=\psi_{F}\circ\text{Tr}_{D/F}$ on $D$ which is to be considered as a
  character of $N$. For a smooth representation $(\pi, V)$ of $G$, the
  space spanned by the set of vectors
  $\{\pi(n)v-\psi(n)v:v\in V, n\in N\}$ is denoted by $V(N,
  \psi)$. The twisted Jacquet module $V_{N,\psi}$ of $V$ is the
  quotient $V/ V(N, \psi)$ considered as a representation of
  ${\rm Stab}_T(\psi)= D^{\times}$. Recall that the
  Jacquet-Langlands lemma says that a vector $v\in V(N, \psi)$ if and
  only if
\[\int_{\mathcal{N}}\psi^{-1}(n)\pi(n)vdn=0,\]
for some compact open subgroup $\mathcal{N}$ of $N$ (see \cite[Lemma
2.33]{Bernsteinzelevinsky0}).  Though the notation for an element of $N$ is the same as that of the index of $D$, it should be clear from the context what it is used for.
\subsection{}
For an irreducible smooth representation $\tau$ of $D^\times$, there
exists an unramified character $\nu_\tau$ such that the normalized
principal series representation
$\tau\nu_\tau^{-1/2}\times \tau\nu_\tau^{1/2}$ of $G$ is reducible of length 2 and
has a unique square-integrable quotient, the generalized Steinberg representation, denoted by ${\rm
  St}(\tau)$. The subrepresentation $\mathrm{Sp}(\tau)$ of $\tau\nu_\tau^{-1/2}\times \tau\nu_\tau^{1/2}$ is called the generalized Speh representation. We have the following short exact sequences of $G$-representations:
\[0\longrightarrow {\rm Sp}(\tau)\longrightarrow
\tau\nu_\tau^{-1/2}\times \tau\nu_\tau^{1/2}
\longrightarrow
{\rm St}(\tau)\rightarrow 0\] and
\[0\longrightarrow {\rm St}(\tau)\longrightarrow
\tau\nu_\tau^{1/2}\times \tau\nu_\tau^{-1/2}
\longrightarrow
{\rm Sp}(\tau)\rightarrow 0.\]
The unramified character $\nu_{\tau}$ is $|\cdot|^{\frac{a(\tau)}{n}}$ where $a(\tau)$ is the length of the segment that determines the Jacquet-Langlands lift of $\tau$. We refer to Tadic for the above results \cite{tad90}. If $\tau$ is tamely ramified of dimension $d$, then $n=a(\tau)d$ and thus  $\nu_{\tau}=|\cdot|^{\frac{1}{d}}$, see \cite[Remark on page 182]{sz05}. For a principal series $\tau_{1}\times\tau_{2}$ of $G$, there is a natural isomorphism of $D^{\times}$-representations \[(\tau_{1}\times\tau_{2})_{N,\psi}\simeq\tau_{1}\otimes\tau_{2},\]
see \cite[Theorem 2.1]{pr00}. We note that $\mathrm{Sp}(\tau)_{N,\psi}\neq 0$ if and only if $\tau$ has dimension $>1$. Finally, if $H\subseteq G$ is a subgroup, then the restriction of a $G$-representation $V$ to $H$ is denoted either by $\res_{H}V$ or by $V|_{H}$.

\section{Proof of the conjecture of D. Prasad}

In a note \cite{pr01}, D. Prasad conjectured that \[\mathrm{Sp}(\tau)_{N,\psi}\simeq\omega_{\tau}\circ\mathrm{Nr}_{D/F}\] as $D^{\times}$-representations when $D$ is the quaternionic division algebra and $\tau$ is a smooth irreducible representation of $D^{\times}$ of dimension $>1$. We first prove this conjecture: 

\begin{theorem}\label{dpthm}
Let $D$ be the quaternionic division algebra over $F$ and let 
$\tau$ be a smooth irreducible representation of $D^\times$ of dimension $>1$. Then \[\mathrm{Sp}(\tau)_{N,\psi}\simeq\omega_{\tau}\circ\mathrm{Nr}_{D/F}\] as $D^{\times}$-representations.
\end{theorem}
\begin{proof}


We first show that  $\mathrm{Sp}(\tau)_{N,\psi}$ is a character of $D^{\times}$, i.e., the space $\mathrm{Sp}(\tau)_{N,\psi}$ is one-dimensional. Denote by $\sigma$ the Jacquet-Langlands lift of $\tau$. Note that $\sigma$ is cuspidal. Consider the segment $\Delta=[\sigma|\cdot|_{F}^{-1/2}, \sigma|\cdot|_{F}^{1/2}]$ and let $\langle \Delta \rangle$ be the irreducible subrepresentation of $\mathrm{GL}_{4}(F)$ associated with the segment $\Delta$ as in \cite[Section 3]{zelevinsky2}. The Jacquet-Langlands correspondence between $G={\rm GL}_2(D)$ and ${\rm GL}_4(F)$ and its extension to the Grothendieck groups of irreducible smooth representations takes the representation ${\rm Sp}(\tau)$ to $\langle \Delta \rangle$. The coefficient of the leading term in the germ expansion of ${\rm Sp}(\tau)$, denoted by $c_{\mathcal{O}}({\rm Sp}(\tau))$, is the dimension of ${\rm Sp}(\tau)_{N, \psi}$, and we have
    $$c_{\mathcal{O}}({\rm Sp}(\tau))=c_{\mathcal{O}'}(\langle\Delta\rangle),$$
    where $\mathcal{O}'$ is the nilpotent orbit of $\mathfrak{gl}_4(F)$ corresponding to the
    partition $(2,2)$ (see \cite[Theorem 2]{dipendragermexpansions} and \cite[Theorem 1.3, 1.6, 1.7]{hv23} for positive characteristic $F$). Using \cite[Proposition 3.4]{zelevinsky2}), we get that $(2,2)$ is the maximal
     element in the Whittaker support of $\langle\Delta\rangle$. So the nilpotent orbit associated with the partition $(2,2)$
     is the maximal nilpotent orbit in the germ expansion of $\langle \Delta \rangle$
     and thus $c_{\mathcal{O}'}(\langle\Delta\rangle)$ is $1$ (see \cite[Theorem I.16 and Chapitre II, II.2]{mwdegenerate}). 

Now the fact that a $D^{\times}$-character factors through the reduced norm $\mathrm{Nr}_{D/F}$ and that $\mathrm{Sp}(\tau)_{N,\psi}\hookrightarrow(\tau_{1}\times\tau_{2})_{N,\psi}\simeq\tau_{1}\otimes\tau_{2}$ gives $\mathrm{Sp}(\tau)_{N,\psi}\simeq\chi\circ\mathrm{Nr}_{D/F}$, where $\chi$ is a character on $F^{\times}$ that is equal to $\omega_{\tau}$ up to a quadratic character. However, the result \cite[Theorem 8.6]{shalika-periodsgantakeda} of Gan and Takeda on Shalika models of Speh representations implies that $\chi=\omega_{\tau}$.
\end{proof}

\begin{remark}
Note that the above argument does not work when $D$ is not the quaternionic division algebra because the non-trivial nilpotent orbit of $\mathfrak{gl}_2(D)$ corresponds to the nilpotent orbit of $\mathfrak{gl}_{2n}(F)$
associated with the partition $(n,n)$ of $2n$, whereas the 
maximal element in the Whittaker support of $\langle \Delta \rangle$ corresponds to the partition $(2,2,\dots, 2)$ of $2n$.
\end{remark}
\section{Further results in the tame case}
To understand the structure of twisted Jacquet modules of generalized Speh representations for arbitrary division algebra, we restrict ourselves from now on to tamely ramified (depth $0$) representations. A generalization of Theorem \ref{dpthm} is obtained for an arbitrary division algebra in the tame case. 
\subsection{Dimension formulae}
Fix an additive character $\psi_{F}:F\rightarrow\mathbb{C}^{\times}$
such that $\psi_{F}$ is non-trivial on $\mathfrak{o}_F$ but trivial on
$\mathfrak{p}_F$. Then $\psi=\psi_{F}\circ\text{Tr}_{D/F}$ is
non-trivial on $\mathfrak{o}_{D}$ and trivial on
$\mathfrak{p}_{D}$. The map
$$\begin{pmatrix}a&b\\\varpi_{D}c&d\end{pmatrix}\mapsto \psi(b)$$
defines a non-trivial character $\psi_{0}$ on the group $I(1)$ factoring through $U(\mathbb{F}_{q^{n}})$. For
any smooth representation $V$ of $I(1)$, the space of
$\psi_{0}$-semi-invariants is
$$V^{I(1), \psi_{0}}=\{v\in V: gv=\psi_{0}(g)v\ \text{for all}\ g\in
I(1)\}.$$ If $V$ is a smooth $G$-representation, then we note that
$V^{I(1),\psi_{0}}$ is stable under the action of $D^{\times}$. This
is because $\psi$ is trivial on $\mathfrak{p}_{F}$ and factors through
${\rm Tr}_{D/F}$.

Now let $\tau_1$ and $\tau_2$ be two irreducible smooth
depth-zero representations of $D^\times$ of dimensions $d_{1}$ and
$d_{2}$ respectively. As vector spaces, we have \begin{align*}
(\tau_{1}\times\tau_{2})^{I(1),\psi_{0}}&=
\mathrm{Hom}_{I(1)}(\psi_{0},\tau_{1}\times\tau_{2})
\\&=\mathrm{Hom}_{I(1)}(\psi_{0},(\tau_{1}\times\tau_{2})^{K(1)})
\\&=\mathrm{Hom}_{I(1)}(\psi_{0},\mathrm{Ind}_{I}^{K}(\tau_{1}\otimes\tau_{2}))
\\&=\mathrm{Hom}_{U(\mathbb{F}_{q^{n}})}(\psi_{0},\mathrm{Ind}_{B(\mathbb{F}_{q^{n}})}^{\mathrm{GL}_{2}(\mathbb{F}_{q^{n}})}(\tau_{1}\otimes\tau_{2})).
\end{align*}
Thus the space $(\tau_{1}\times\tau_{2})^{I(1),\psi_{0}}$ has dimension
$d_{1}d_{2}$. Also, the space $(\tau_1\times \tau_2)_{N, \psi}$ has dimension $d_{1}d_{2}$ because $(\tau_{1}\times\tau_{2})_{N,\psi}\simeq\tau_{1}\otimes\tau_{2}$ by \cite[Theorem 2.1]{pr00}. In fact, we now show that these two $D^{\times}$-representations are naturally isomorphic to each other:
\begin{theorem}\label{main}
  The restriction of the natural map $\tau_1\times \tau_2\rightarrow (\tau_1\times
  \tau_2)_{N, \psi}$ to the subspace $(\tau_1\times \tau_2)^{I(1), \psi_{0}}$:
  \begin{equation*}
    (\tau_1\times \tau_2)^{I(1), \psi_{0}}\rightarrow (\tau_1\times
    \tau_2)_{N, \psi}\end{equation*} is an isomorphism of $D^{\times}$-representations. 
\end{theorem}

Before we begin proving the theorem, we need a couple of lemmas. For an
integer $r$, let 
$N(r)=\begin{pmatrix}1&\mathfrak{p}^{r}_{D}\\0&1\end{pmatrix}$. Note
$\psi|_{N(0)}=\psi_{0}|_{N(0)}$.
\begin{lemma}\label{ssupp}
  Let $f$ be a non-zero element of $(\tau_1\times \tau_2)^{I(1),
    \psi_{0}}$, then we have
  $$\int_{N(0)}\psi^{-1}(n)f(sn)dn=\mathrm{vol}(N(0))f(s)\neq 0.$$
\end{lemma}
\begin{proof}
The first equality is clear because $f(sn)=(nf)(s)=\psi(n)f(s)$ for $n\in N(0)$. Now, the function $f$ is non-zero if and only if $f|_{K}$ is so. We have
  $K=I\sqcup IsI$. Observe that $f(1)=0$ because $\psi|_{N(0)}$ is
  non-trivial. From this, we get that $f(i)=0$, for all 
  $i\in I=\left(
  \begin{smallmatrix}\mathbb{F}_{q^{n}}^{\times}&0\\0&\mathbb{F}_{q^{n}}^{\times}\end{smallmatrix}\right)I(1)$. The double coset
  $IsI$ is equal to the set $(I\cap B)sI(1)$. If $f(s)=0$, then the
  function $f$ is identically zero on the double coset $IsI$, and
  hence on $K$. Thus $f(s)\neq 0$.
\end{proof}
\begin{lemma}\label{nbds}
For any smooth representation $(\pi,V)$ of $N$, and $v\in V$, the image of
$v$ in $V_{N, \psi}$ is non-zero if and only if
$$\int_{N(-r)}\psi^{-1}(n)\pi(n)vdn$$
is non-zero for all $r>>0$. 
\end{lemma}
\begin{proof}
Assume that the image of a vector $v\in V$ in $V_{N, \psi}$ is zero. Then there
exists a compact open subgroup $\mathcal{N}$ such that
$$\int_{\mathcal{N}}\psi^{-1}(n)\pi(n)vdn=0.$$
Since $\{N(-r):r>0\}$ is an increasing filtration of $N$, there exists
an $r$ such that $ \mathcal{N}\subset N(-r)$. Thus,
$$\int_{N(-r)}\psi^{-1}(n)\pi(n)vdn=
\sum_{g\in N(-r)/\mathcal{N}}\psi^{-1}(g)\pi(g)\int_{\mathcal{N}}\psi^{-1}(n)\pi(n)vdn=0,$$
for all $r$ such that $ \mathcal{N}\subset N(-r)$. Conversely, if the
above integral is zero for any $r>0$, then the image of $v$ in
$V_{N, \psi}$ is zero.
\end{proof}
\begin{proof}[Proof of Theorem \ref{main}.]
For
any positive integer $r$ and $u\in \mathfrak{o}_D^\times$, we have
the following matrix identity:
\begin{equation}\label{snidentity}
\begin{pmatrix}0&1\\-1&0\end{pmatrix}
\begin{pmatrix}1&\varpi_D^{-r}u\\0&1\end{pmatrix}=
\begin{pmatrix}0&1\\-1&-\varpi_D^{-r}u\end{pmatrix}=
\begin{pmatrix}-\varpi_D^{-r}u&-1\\0&-\varpi_D^r\end{pmatrix}^{-1}
\begin{pmatrix}1&0\\\varpi_D^r&u\end{pmatrix}.
\end{equation}
Let $f\in (\tau_1\times \tau_2)^{I(1), \psi_{0}}$ be a non-zero
function and $f_r:=\int_{N(-r)}\psi^{-1}(n)\pi(n)fdn$ where $\pi=\tau_{1}\times\tau_{2}$. For $a\in\mathfrak{p}_{D}^{-r}$, let $\overline{a}$ denote its class $a+\mathfrak{p}_{D}^{-r}$ in $\mathfrak{p}_D^{-r}/\mathfrak{p}_D^{-r+1}$. Then
\begin{align*}
  f_{r}(s)=\int_{N(-r)}\psi^{-1}(n)f(sn)dn=
  &\int_{\mathfrak{p}_D^{-r}}\psi^{-1}(y)f\left(s \begin{pmatrix}1&y\\0&1\end{pmatrix}\right)dy\\
  =&\sum_{\overline{a}\in\mathfrak{p}_D^{-r}/\mathfrak{p}_D^{-r+1}}
     \int_{\mathfrak{p}_D^{-r+1}}\psi^{-1}(a+y)f\left(s \begin{pmatrix}1&(a+y)\\0&1\end{pmatrix}\right)dy\\
  =&f_{r-1}(s)+\sum_{\overline{a}\neq\overline{0}}
     \int_{\mathfrak{p}_D^{-r+1}}\psi^{-1}(a+y)f\left(s \begin{pmatrix}1&(a+y)\\0&1\end{pmatrix}\right)dy.
\end{align*}
If $r>0$, then using the identity \eqref{snidentity} and that
$f(i)=0$ for $i\in I$ (cf. the proof of Lemma \ref{ssupp}), we get that
\[\text{$\int_{\mathfrak{p}_D^{-r+1}}
    \psi^{-1}(a+y)f\left(s \begin{pmatrix}1&(a+y)\\0&1\end{pmatrix}\right)dy=0$
    for $\overline{a}\neq\overline{0}$.}\] 
Thus, we obtain $f_{r}(s)=f_{r-1}(s)$ for all $r>0$. By Lemma
\ref{ssupp}, we get that $f_r$ is non-zero for all $r\geq 0$. Hence,
by Lemma \ref{nbds}, the natural map
\begin{equation}\label{natural1}
  (\tau_1\times \tau_2)^{I(1), \psi_{0}}\rightarrow
  (\tau_1\times \tau_2)_{N, \psi}\end{equation}
is injective. As both sides have the same dimension, the map in \eqref{natural1} is an isomorphism. 
\end{proof}

\begin{proposition}\label{semiinvariantsforspst}
  Let $\tau$ be a tamely ramified irreducible representation of
  $D^\times$.  The natural maps
  $${\rm Sp}(\tau)^{I(1), \psi_{0}}\rightarrow {\rm Sp}(\tau)_{N, \psi}$$
  and
  $${\rm St}(\tau)^{I(1), \psi_{0}}\rightarrow {\rm St}(\tau)_{N, \psi}$$
  are isomorphisms.
\end{proposition}
\begin{proof}
  We have the following commutative diagrams:
\begin{equation*}
	\begin{tikzcd}
          0 \arrow[r] & {\rm Sp}(\tau)_{N,\psi} \arrow[r] &
          (\tau\nu_{\tau}^{-1/2}\times\tau{\nu}_{\tau}^{1/2})_{N,\psi}
          \arrow[r] & {\rm St}(\tau)_{N,\psi} \arrow[r] & 0 \\ 0
          \arrow[r] & {\rm Sp}(\tau)^{I(1),\psi_{0}}
          \arrow[r]\arrow[u,"f"] &
          (\tau\nu_{\tau}^{-1/2}\times\tau{\nu}_{\tau}^{1/2})^{I(1),\psi_{0}}
          \arrow[r]\arrow[u,"g"] & {\rm St}(\tau)^{I(1),\psi_{0}}
          \arrow[r]\arrow[u,"h"] & 0
	\end{tikzcd}
\end{equation*}
and
\begin{equation*}
	\begin{tikzcd}
          0 \arrow[r] & {\rm St}(\tau)_{N,\psi} \arrow[r] &
          (\tau\nu_{\tau}^{1/2}\times\tau{\nu}_{\tau}^{-1/2})_{N,\psi}
          \arrow[r] & {\rm Sp}(\tau)_{N,\psi} \arrow[r] & 0 \\ 0
          \arrow[r] & {\rm St}(\tau)^{I(1),\psi_{0}}
          \arrow[r]\arrow[u,"h"] &
          (\tau\nu_{\tau}^{1/2}\times\tau{\nu}_{\tau}^{-1/2})^{I(1),\psi_{0}}
          \arrow[r]\arrow[u,"g'"] & {\rm Sp}(\tau)^{I(1),\psi_{0}}
          \arrow[r]\arrow[u,"f"] & 0
	\end{tikzcd}
\end{equation*}
where $f$, $g$, $g'$ and $h$ are the natural maps. Since $g$ and $g'$ are
isomorphisms from Theorem \ref{main}, we get that $f$ is injective and
$h$ is surjective from the first diagram and $f$ is surjective and $h$
is injective from the second diagram. 
  \end{proof}

\begin{corollary}\label{dimformulae}
  Let $\tau=\mathrm{Ind}_{D^{\times}_{m}D(1)}^{D^{\times}}\tilde{\theta}$ be a
  $d$-dimensional tamely ramified irreducible representation of $D^\times$. We
  then have
    $$\dim_{\mathbb{C}}{\rm St}(\tau)_{N, \psi}=\dfrac{d(d-1)}{2}+d \hspace{3mm}\text{and}\hspace{3mm}\dim_{\mathbb{C}}{\rm Sp}(\tau)_{N, \psi}=\dfrac{d(d-1)}{2}.$$
  \end{corollary}
  \begin{proof}
 From the work of Minguez and Secherre \cite{ms14}, we find that as
    $K$-representations \begin{align*}
      &\mathrm{St}(\tau)^{K(1)}\simeq\bigoplus_{\substack{i,j\in\mathbb{Z}/d\mathbb{Z}\\i\neq j}}\mathrm{Ind}_{I}^{K}(\tilde{\theta}^{q^{i}}\otimes\tilde{\theta}^{q^{j}})\oplus\bigoplus_{i\in\mathbb{Z}/d\mathbb{Z}}\mathrm{st}(\tilde{\theta}^{q^{i}})
        \hspace{2mm}\text{and}\\&\mathrm{Sp}(\tau)^{K(1)}\simeq\bigoplus_{\substack{i,j\in\mathbb{Z}/d\mathbb{Z}\\i\neq j}}\mathrm{Ind}_{I}^{K}(\tilde{\theta}^{q^{i}}\otimes\tilde{\theta}^{q^{j}})\oplus\bigoplus_{i\in\mathbb{Z}/d\mathbb{Z}}\tilde{\theta}^{q^{i}}\circ\mathrm{det}(\overline{\hspace{0.5mm}\cdot\hspace{0.5mm}}),
        \end{align*} where $\mathrm{det}(\overline{\hspace{0.5mm}\cdot\hspace{0.5mm}})$ is the composition of the determinant character of $\mathrm{GL}_{2}(\mathbb{F}_{q^{n}})$ and the natural surjection $K\twoheadrightarrow\mathrm{GL}_{2}(\mathbb{F}_{q^{n}})$, and $\tilde{\theta}^{q^{i}}\circ\mathrm{det}(\overline{\hspace{0.5mm}\cdot\hspace{0.5mm}})$ and $\mathrm{st}(\tilde{\theta}^{q^{i}})$ are the two simple factors of the reducible induction $\mathrm{Ind}_{I}^{K}(\tilde{\theta}^{q^{i}}\otimes\tilde{\theta}^{q^{i}})$ (see \cite[Lemma 4.5]{ns23}).  Hence, \[\dim_{\mathbb{C}}{\rm
                            St}(\tau)^{I(1), \psi_{0}}=\dim_{\mathbb{C}}\mathrm{Hom}_{I(1)}(\psi_{0},
                          \mathrm{St}(\tau)^{K(1)})=\dfrac{d(d-1)}{2}+d\]
    and
    \[\dim_{\mathbb{C}}{\rm Sp}(\tau)^{I(1), \psi_{0}}=
      \dim_{\mathbb{C}}\mathrm{Hom}_{I(1)}(\psi_{0},\mathrm{Sp}(\tau)^{K(1)})=
      \dfrac{d(d-1)}{2}.\]
    The corollary now follows from Proposition \ref{semiinvariantsforspst}. 
  \end{proof}
  \subsection{The $D^\times$-action on the twisted Jacquet module}
Let $\tau=\mathrm{Ind}_{D^{\times}_{m}D(1)}^{D^{\times}}\tilde{\theta}$ be a
  $d$-dimensional tamely ramified irreducible representation of $D^\times$. Using Proposition \ref{semiinvariantsforspst}, we now find the explicit structure of the $D^{\times}$-representation $\mathrm{Sp}(\tau)_{N,\psi}$. The analysis depends on the parity of $d$. 

The representation $\tau$ restricted to the subgroup $\mathbb{F}_{q^{n}}^{\times}$ decomposes as the sum of characters $\bigoplus_{i\in\mathbb{Z}/d\mathbb{Z}}\tilde{\theta}^{q^{i}}$ and the $\varpi_{D}$-action maps  the underlying space of $\tilde{\theta}^{q^{i}}$ to that of $\tilde{\theta}^{q^{i-1}}$. Consider now $\tau\otimes\tau$ as a representation of $B$. Letting $K(1)$ act trivially, one extends the action of $B\cap KD^{\times}$ on $\tau\otimes\tau$ to $ID^{\times}=(B\cap KD^{\times})K(1)$. Then, as a representation of $ID^{\times}$,
 \[\tau\otimes\tau=\bigoplus_{j\in\mathbb{Z}/d\mathbb{Z}}W_{j},\] 
where $W_j$ is an irreducible
representation of $ID^\times$ such that \[\res_{I}W_j=\bigoplus_{i\in\mathbb{Z}/d\mathbb{Z}}\tilde{\theta}^{q^i}\otimes \tilde{\theta}^{q^{i+j}}.\]
So the space of $K(1)$-invariants of the principal series $\tau\nu_\tau^{-1/2}\times
\tau\nu_\tau^{1/2}$ as a $KD^\times$-representation is isomorphic to
$$\ind_{ID^\times}^{KD^\times}(\tau\otimes \tau)=\ind_{ID^\times}^{KD^\times}\left(\oplus_{j\in
	\mathbb{Z}/d\mathbb{Z}}W_j\right). $$
\begin{lemma}\label{oddcaselemma}
  Let $j\in \mathbb{Z}/d\mathbb{Z}$. If $2j\neq 0$, then the
  representation $\ind_{ID^\times}^{KD^\times} W_j$ is irreducible; otherwise
  it has two distinct irreducible subrepresentations $\rho_1$ and $\rho_2$.
  When $j\neq 0$ (and $2j=0$), we have $\res_K\rho_1\simeq \res_K\rho_2$. Moreover,
  $$\ind_{ID^\times}^{KD^\times}W_j\simeq \ind_{ID^\times}^{KD^\times}W_{-j}$$
  for all $2j\neq 0$. 
  \end{lemma}
\begin{proof}
  Applying the Mackey decomposition, we get that
  \[{\rm Hom}_{KD^\times}(\ind_{ID^\times}^{KD^\times}W_{j},
    \ind_{ID^\times}^{KD^\times}W_{j'}) ={\rm
      Hom}_{T_0D^\times}(W_{j}, W_{j'})\oplus {\rm
      Hom}_{T_0D^\times}(W_j, W_{j'}^s),\] where $W^{s}_{j}$ is a $T_{0}D^{\times}$-representation on the space $W_{j}$ equipped with the $T_{0}D^{\times}$-action conjugated by $s=\left(
     \begin{smallmatrix}0&1\\-1&0\end{smallmatrix}\right)$. For
  $j\in \mathbb{Z}/d\mathbb{Z}$, the representations $W_j$ of
  $T_0D^\times$ are distinct and irreducible. As $W_j^s$ is
  equal to $W_{-j}$, the lemma follows.
\end{proof}

Let $\{e_i: i\in\mathbb{Z}/d\mathbb{Z}\}$ be a basis of
$\tau$ consisting of functions $e_{i}:D^{\times}\rightarrow\mathbb{C}$ such that $\text{supp}(e_{i})=\mathfrak{o}_D^{\times}\rtimes\varpi_{D}^{d\mathbb{Z}}\varpi_{D}^{i}$ and $e_{i}(\varpi_{D}^{i})=1$. The (diagonal) character of $\mathfrak{o}_D^\times$ on the $1$-dimensional space spanned by the vector $e_{i}\otimes e_{i+j}$ is $(\tilde{\theta}^{1+q^{j}})^{q^{i}}$. 
For $j\in\mathbb{Z}/d\mathbb{Z}$, $\res_{D^\times}W_{j}$ is a representation of $D^{\times}$ such that \[\res_{\mathfrak{o}_D^\times}W_{j}\simeq\bigoplus_{i\in\mathbb{Z}/d\mathbb{Z}}(\tilde{\theta}^{1+q^{j}})^{q^i}\] with $\varpi_{D}$ mapping $e_{i}\otimes e_{i+j}$ to $e_{i-1}\otimes e_{i-1+j}$ for $i\neq 0$ and $e_{0}\otimes e_{j}$ to $\tilde{\theta}(\varpi_{D}^{2d})e_{-1}\otimes e_{-1+j}$.
\subsubsection{$d=\mathrm{dim}_{\mathbb{C}}(\tau)$ is odd}
\begin{proposition}\label{oddcaseprop}
  If $d=\mathrm{dim}_{\mathbb{C}}(\tau)$ is odd, then 
  the $D^{\times}$-representation ${\rm Sp}(\tau)_{N, \psi}$ is
  isomorphic to $\bigwedge^2\tau$. 
\end{proposition}
\begin{proof}
Let $S\subseteq\mathbb{Z}/d\mathbb{Z}$ be the subset consisting of elements $j$'s defined by the condition that $2j\neq 0$ and $j\in S$ if and only if $-j\notin S$. By the proof of Corollary \ref{dimformulae} and Lemma \ref{oddcaselemma}, we have as $KD^{\times}$-representations,
\[{\rm Sp}(\tau)^{K(1)}\simeq V\oplus \bigoplus_{j\in S}\ind_{ID^\times}^{KD^\times}W_j,\] where $\res_{K}V=\bigoplus_{i\in\mathbb{Z}/d\mathbb{Z}}\tilde{\theta}^{q^{i}}\circ\mathrm{det}(\overline{\hspace{0.5mm}\cdot\hspace{0.5mm}})
$. Considering the $\psi_{0}$-semi-invariants for the action of $I(1)$, we get that
\[{\rm Sp}(\tau)^{I(1), \psi_{0}}\simeq \bigoplus_{j\in S}{\rm
    Res}_{D^\times} W_j.\] The set
$\{e_i\otimes e_j-e_j\otimes e_i:\text{unordered pairs}\hspace{2mm} (i,j)\in(\mathbb{Z}/d\mathbb{Z})^{2}\hspace{2mm}\text{with}\hspace{2mm} i\neq j\}$ is a basis for
$\bigwedge^2(\tau)$, whereas the space $W_j$ is spanned by vectors
$e_i\otimes e_{i+j}, i\in\mathbb{Z}/d\mathbb{Z}$. The map
\[e_i\otimes e_{i+j}\mapsto e_i\otimes e_{i+j}-e_{i+j}\otimes e_i\]
defines an isomorphism of $\bigoplus_{j\in S}\res_{D^\times}W_j$ with $\bigwedge^2\tau$.
\end{proof}
\subsubsection{$d=\mathrm{dim}_{\mathbb{C}}(\tau)$ is even} 
\hspace{2mm}

As in the proof of Proposition \ref{oddcaseprop}, we have 
\begin{equation}\label{sptjmdescription}
	\text{${\rm Sp}(\tau)_{N,\psi}\simeq{\rm Sp}(\tau)^{I(1), \psi_{0}}\simeq X\oplus\bigoplus_{j\in S}{\rm
			Res}_{D^\times} W_j$}
\end{equation} where $X$ is a subrepresentation of $\res_{D^{\times}}W_{\frac{d}{2}}$ such that \[\res_{\mathfrak{o}_D^\times}X\simeq\bigoplus_{i=0}^{\frac{d}{2}-1}(\tilde{\theta}^{1+q^{\frac{d}{2}}})^{q^i}.\]
Suppose $(\tilde{\theta}^{1+q^{\frac{d}{2}}})^{q^{k}}=\tilde{\theta}^{1+q^{\frac{d}{2}}}$ with $kk'=\frac{d}{2}$. Then by Frobenius reciprocity, the $D^{\times}$-representation $X$ is a sum of $k'$ copies of the induction of $\tilde{\theta}^{1+q^{\frac{d}{2}}}$ from the index $k$ subgroup $\mathfrak{o}_D^{\times}\rtimes\varpi_{D}^{k\mathbb{Z}}$, that is,
\[X\simeq\ind_{\mathfrak{o}_D^{\times}\rtimes\varpi_{D}^{k\mathbb{Z}}}^{D^{\times}}(\tilde{\theta}^{1+q^{\frac{d}{2}}})\oplus\ldots\oplus\ind_{\mathfrak{o}_D^{\times}\rtimes\varpi_{D}^{k\mathbb{Z}}}^{D^{\times}}(\tilde{\theta}^{1+q^{\frac{d}{2}}}).\] When $k=1$, the underlying space of the character $\tilde{\theta}^{1+q^{\frac{d}{2}}}$ is itself stable under the action of $\varpi_{D}$ and one needs to analyze this action. We do this for $d=2$ in the remaining part of this subsection.

Let $d=\mathrm{dim}_{\mathbb{C}}(\tau)=2$ from now on. As $d=2$, the index of $D$ is $n=2m$ and the set $S$ is empty. From \eqref{sptjmdescription} or by Corollary \ref{dimformulae}, we know that $\mathrm{Sp}(\tau)_{N,\psi}$ is a character of $D^{\times}$. The following theorem precisely describes this character generalizing Theorem \ref{dpthm}. 
  \begin{theorem}\label{explicitstructure}
    The $D^\times$-representation ${\rm Sp}(\tau)_{N, \psi}$ is the character 
    $(\theta\circ {\rm Nr}_{D/F})\mu_{(-1)^{m+1}}$. 
  \end{theorem}
  \begin{proof}
  	As $d=2$, the $KD^{\times}$-representation ${\rm Sp}(\tau)^{K(1)}$ is a sum of two representations $V$ and $\rho$ such that $\res_{K}V\simeq\left(\tilde{\theta}\circ\mathrm{det}(\overline{\hspace{0.5mm}\cdot\hspace{0.5mm}})\right)\oplus\left(\tilde{\theta}^{q}\circ\mathrm{det}(\overline{\hspace{0.5mm}\cdot\hspace{0.5mm}})\right)$ and $\res_{K}\rho\simeq\ind_{I}^{K}(\tilde{\theta}^{q}\otimes\tilde{\theta})$ (see the proof of Corollary \ref{dimformulae} and Lemma \ref{oddcaselemma}).
    Let $f$ be a non-zero function in $\mathrm{Sp}(\tau)^{K(1)}$ such
    that $kf=\tilde{\theta}({\rm det}(\overline{k}))f$ for all $k\in K$ and let
   $$t:=\begin{pmatrix}\varpi_D&0\\0&1\end{pmatrix}.$$
   Note that $tit^{-1}\in K$ for $i\in I$ and thus
   $it^{-1}f=\tilde{\theta}({\rm det}(\overline{tit^{-1}}))t^{-1}f$, which
   implies that
   $t^{-1}f\in {\rm Sp}(\tau)^{I(1)}$ and $I$ acts on $t^{-1}f$ via the character $\tilde{\theta}^{q}\otimes\tilde{\theta}$. It follows that the $K$-representation
   $\langle K\cdot t^{-1}f\rangle$ generated by $t^{-1}f$ is $\res_{K}\rho$. We are interested in the $D^{\times}$-representation
   on the space $\langle K\cdot t^{-1}f\rangle^{I(1),\psi_{0}}$. 
   
   The
   Frobenius reciprocity map
   $$\Phi: {\rm Ind}_{I}^K(\tilde{\theta}^q\otimes \tilde{\theta})\rightarrow
   \langle K\cdot t^{-1}f\rangle$$ is an isomorphism of $K$-representations where   \[\Phi(\varphi)=\sum_{k\in\{1,sn_{x}\}}\varphi(k^{-1})kt^{-1}f.\] Here,
   $\{1,sn_{x}=\left(
     \begin{smallmatrix}0&1\\-1&0\end{smallmatrix}\right)
   \left(\begin{smallmatrix}1&[x]\\0&1\end{smallmatrix}\right):x\in\mathbb{F}_{q^{2m}}\}$
   is a set of representatives for $I\backslash K$. Let
   $\mathbb{1}_{I}\in{\rm Ind}_{I}^K(\tilde{\theta}^q\otimes \tilde{\theta})$ be the
   function such that $\Phi(\mathbb{1}_{I})=t^{-1}f$. It is the function supported on $I$ mapping $1$ (of $K$) to $1$.
   Using the operator
   \[T(\varphi)(k)=\dfrac{\theta(-1)^{m+1}\theta(\varpi_{F})}{q^m}
     \sum_{y\in\mathbb{F}_{q^{2m}}}\varphi(sn_{y}\varpi_{D}k\varpi_{D}^{-1})\]
   we make ${\rm Ind}_{I}^K(\tilde{\theta}^q\otimes \tilde{\theta})$ into a
   representation of $KD^\times$ such that $\varpi_D$ acts by $T$. We
   claim that $\Phi$ is then an isomorphism of
   $KD^{\times}$-representations. Indeed, we note that
   $T^{2m}=\theta(\varpi_{F})^{2}{\rm Id}$ and $T$ corresponds to an
   intertwining operator in
   $\mathrm{Hom}_{K}({\rm Ind}_{I}^K(\tilde{\theta}^q\otimes \tilde{\theta}),{\rm
     Ind}_{I}^K(\tilde{\theta}\otimes \tilde{\theta}^{q}))\simeq\mathbb{C}$. As
   $\varpi_{D}^{2m}$ acts on $\langle K\cdot t^{-1}f\rangle$ by the
   scalar-multiplication by $\theta(\varpi_{F})^{2}$,
   there exists a scalar (in fact, an $m$-th root of unity) $\epsilon$ such that
\[\Phi(T(\mathbb{1}_I))= \epsilon\varpi_Dt^{-1}f.\]
     Expanding the left-hand side of the above, we find that
     \begin{align*}
       \Phi(T(\mathbb{1}_I))=
       &\sum_{k\in\{1,sn_{x}\}}T(\mathbb{1}_{I})(k^{-1})kt^{-1}f\\
       =&\dfrac{\theta(-1)^{m+1}\theta(\varpi_{F})}{q^m}
          \sum_{k\in\{1,sn_{x}\}}
          \sum_{y\in\mathbb{F}_{q^{2m}}}\mathbb{1}_{I}(sn_{y}\varpi_{D}k^{-1}
          \varpi_{D}^{-1})kt^{-1}f\\
       =&\dfrac{\theta(-1)^{m+1}\theta(\varpi_{F})}{q^m}
          \sum_{x\in\mathbb{F}_{q^{2m}}}
          \sum_{y\in\mathbb{F}_{q^{2m}}}\mathbb{1}_{I}(sn_{y}\varpi_{D}n_{x}s
          \varpi_{D}^{-1})sn_{-x}t^{-1}f\\
       =&\dfrac{\theta(-1)^{m+1}\theta(\varpi_{F})}{q^m}
          \sum_{x\in\mathbb{F}_{q^{2m}}}
          \sum_{y\in\mathbb{F}_{q^{2m}}}\mathbb{1}_{I}(sn_{y+x^{q}}s)st^{-1}f\\
       =&\dfrac{\theta(-1)^{m+1}\theta(\varpi_{F})}{q^m}
          \sum_{x\in\mathbb{F}_{q^{2m}}}
          st^{-1}f
      =\theta(-1)^{m+1}\theta(\varpi_{F})q^{m}st^{-1}f.
     \end{align*}
     Thus, we have
     $
     \theta(-1)^{m+1}\theta(\varpi_{F})q^{m}st^{-1}f
     =\epsilon\varpi_Dt^{-1}f$. As $s\in K$ and $\mathrm{det}(\overline{s})=1$, $sf=f$. Hence $
     \theta(-1)^{m+1}\theta(\varpi_{F})q^{m}st^{-1}sf
     =\epsilon\varpi_Dt^{-1}f$. Evaluating
     both sides on $1$, we obtain 
     \[\theta(-1)^{m+1}\theta(\varpi_{F})q^{m}f(st^{-1}s)=
         \epsilon f(\varpi_Dt^{-1}).\] Thus
$$\theta(-1)^{m+1}\theta(\varpi_{F})q^{m}
f\left(\begin{pmatrix}1&0\\0&\varpi_D^{-1}
       \end{pmatrix}\right)
     = \epsilon
     f\left(\begin{pmatrix}1&0\\0&\varpi_D\end{pmatrix}\right).$$
Using that $f\in \tau\nu_{\tau}^{-1/2}\times \tau\nu_{\tau}^{1/2}=\ind_{B}^{G}(\tau|\cdot|^{1/4}\otimes\tau|\cdot|^{-1/4})$, we get
     $$\theta(-1)^{m+1}\theta(\varpi_{F})q^{m}
     (\mathrm{Id}\otimes \tau(\varpi_D^{-1})|\varpi_D|^{1/4})f(1)=
     \epsilon(\mathrm{Id}\otimes \tau(\varpi_D)|\varpi_D|^{-1/4})f(1),$$ and
     using $|\varpi_{D}|=q^{-2m}$ and
     $\tau(\varpi_{D}^{-1})
     =\tau(\varpi_D)\theta(-1)^{m+1}\theta(\varpi_{F})^{-1}$,
     we conclude
      $$ q^{m/2}(\mathrm{Id}\otimes \tau(\varpi_D))f(1)=
      \epsilon q^{m/2}(\mathrm{Id}\otimes \tau(\varpi_D))f(1).$$ Hence,
      $\epsilon=1$ and
      \[\Phi(\varpi_{D}\mathbb{1}_{I})
        =\Phi(T(\mathbb{1}_{I}))=\varpi_{D}t^{-1}f=\varpi_{D}\Phi(\mathbb{1}_{I}).\]
      Thus the $KD^\times$-representation
      $\langle K\cdot t^{-1}f\rangle$ is isomorphic to
      ${\rm Ind}_{I}^K(\tilde{\theta}^{q}\otimes \tilde{\theta})$ with the $\varpi_{D}$-action on the latter given by $T$. 
      
      By \cite[Proposition 2.0.10]{gar14}, the
      $D^\times$-representation on the space
      $\langle K\cdot t^{-1}f\rangle^{I(1), \psi_0}$ is isomorphic to
      $\tilde{\theta}^{q+1}\mu_c$, where
      \begin{equation*}
        c=\dfrac{\theta(-1)^{m+1}\theta(\varpi_{F})}{q^m}
        G(\tilde{\theta}^{q-1}, \psi_0),
      \end{equation*} and $G(\tilde{\theta}^{q-1},\psi_{0})=\sum_{x\in\mathbb{F}_{q^{2m}}}\tilde{\theta}^{q-1}(x)\psi_{0}(x)$ is the Gauss sum. On the underlying space of the character $\tilde{\theta}^{q+1}\mu_{c}$, $\mathfrak{o}_{D}^{\times}$ acts via the character $\tilde{\theta}^{q+1}$ and $\varpi_{D}$ acts as the scalar-multiplication by $c$. To compute the constant $c$, we need to compute the Gauss sum. Note that, in the Gauss sum, $\psi_{0}$ is viewed as a non-trivial additive character on $\mathbb{F}_{q^{2m}}$ factoring as $\psi_{\mathbb{F}_{q}}\circ\mathrm{Tr}_{\mathbb{F}_{q^{2m}}/\mathbb{F}_{q}}$ where $\psi_{\mathbb{F}_{q}}=\psi_{F}|_{\mathfrak{o}_{F}}$. By Hasse-Davenport lifting relation, \[G(\tilde{\theta}^{q-1},\psi_{0})=(-1)^{m+1}G(\theta^{q-1},\psi_{\mathbb{F}_{q}}\circ\mathrm{Tr}_{\mathbb{F}_{q^{2}}/\mathbb{F}_{q}})^{m}.\] To compute $G(\theta^{q-1},\psi_{\mathbb{F}_{q}}\circ\mathrm{Tr}_{\mathbb{F}_{q^{2}}/\mathbb{F}_{q}})=\sum_{x\in\mathbb{F}_{q^{2}}}\theta^{q-1}(x)\psi_{\mathbb{F}_{q}}(\mathrm{Tr}_{\mathbb{F}_{q^{2}}/\mathbb{F}_{q}}(x))$, let us fix a set $\{x_i\}$ of coset
    representatives for $\mathbb{F}_{q^2}^\times/\mathbb{F}_q^\times$.
    We abbreviate ${\rm Tr}_{\mathbb{F}_{q^2}/ \mathbb{F}_q}$ as
    ${\rm Tr}$. Then
    \begin{align*}
      G(\theta^{q-1},\psi_{\mathbb{F}_{q}}\circ\mathrm{Tr})=&\sum_{x_i}\sum_{y\in
                             \mathbb{F}_q^\times}\theta^{q-1}(x_iy)\psi_{\mathbb{F}_{q}}(\mathrm{Tr}(x_iy))\\
      =&\sum_{x_i}\sum_{y\in
                             \mathbb{F}_q^\times}\theta^{q-1}(x_i)\psi_{\mathbb{F}_{q}}({\rm
         Tr}(x_i)y)\\
      =&\sum_{x_i, {\rm Tr}(x_i)=0}(q-1)\theta^{q-1}(x_i)+
         \sum_{x_i, {\rm Tr}(x_i)\neq 0}-\theta^{q-1}(x_i)\\
       =&\sum_{x_i, {\rm Tr}(x_i)=0}q\theta^{q-1}(x_i)-
          \sum_{x_i}\theta^{q-1}(x_i)\\
          =&\sum_{x_i, {\rm Tr}(x_i)=0}q\theta^{q-1}(x_i). 
      \end{align*}
      Note that if $\mathrm{Tr}(x)=\mathrm{Tr}(y)=0$ for some
      $x,y\in\mathbb{F}_{q^2}^{\times}$ then $x$ and $y$ belong to the
      same coset in
      $\mathbb{F}_{q^2}^{\times}/\mathbb{F}_{q}^{\times}$, i.e.,
      $\frac{x}{y}\in\mathbb{F}_{q}^{\times}$. This is clear because
      if $\mathrm{Tr}(x)=\mathrm{Tr}(y)=0$ then
      $(\frac{x}{y})^{q-1}=\frac{x^{q}y}{x
        y^{q}}=\frac{-xy}{-xy}=1$. There always exists an element
      $x_{0}\in\mathbb{F}_{q^2}^{\times}$ with $\mathrm{Tr}(x_{0})=0$: $x_0=1$ if $p=2$; otherwise
      for $\mathbb{F}_{q^2}^{\times}=\langle\alpha\rangle$,
      $x_{0}=\alpha^{\frac{q+1}{2}}$. Hence,
      $\mathbb{F}_{q^2}^{\times}/\mathbb{F}_{q}^{\times}$ has a unique
      coset of trace $0$ elements. Therefore
      \[G(\theta^{q-1},\psi_{\mathbb{F}_{q}}\circ\mathrm{Tr})=q\theta(x_{0}^{q}x_{0}^{-1})
        =q\theta(-x_{0}x_{0}^{-1})=q\theta(-1).\]
      Thus, \[G(\tilde{\theta}^{q-1},\psi_{0})=(-1)^{m+1}q^{m}\theta(-1)^{m} \hspace{2mm} \mathrm{and} \hspace{2mm} c=(-1)^{m+1}\theta(-\varpi_{F}).\] 
      
 It now remains to show that $\tilde{\theta}^{q+1}\mu_{c}$ is the same as the character $\theta\circ{\rm Nr}_{D/F}$ multiplied with the unramified character $\mu_{(-1)^{m+1}}$. To see this, note that \[c=(-1)^{m+1}\theta(-\varpi_{F})=(-1)^{m+1}\theta((-1)^{2m+1}\varpi_{F})=(-1)^{m+1}(\theta\circ{\rm Nr}_{D/F})(\varpi_{D}),\] and for $x\in \mathbb{F}_{q^{2m}}^{\times}$ \begin{align*}
 	\tilde{\theta}^{q+1}(x)=\tilde{\theta}(x)^{q+1}&=(\theta\circ{\rm Nr}_{\mathbb{F}_{q^{2m}}/\mathbb{F}_{q^{2}}})(x)^{q+1}\\&=\theta(x^{1+q^{2}+\ldots+q^{2(m-1)}})^{q+1}\\&=\theta(x^{1+q+q^{2}+\ldots+q^{2m-1}})\\&=(\theta\circ{\rm Nr}_{\mathbb{F}_{q^{2m}}/\mathbb{F}_{q}})(x)=(\theta\circ{\rm Nr}_{D/F})(x).
 \end{align*} Hence $\tilde{\theta}^{q+1}\mu_{c}=(\theta\circ{\rm Nr}_{D/F})\mu_{(-1)^{m+1}}$.
       \end{proof}
       
\begin{remark}
 Note that $\omega_{\tau}=\theta^{m}$. Thus the above theorem recovers Theorem \ref{dpthm} in the case of quaternionic division algebra (i.e., $n=2$ and $m=1$).  
\end{remark}
\begin{remark}
We remark that in contrast with odd $d$, for $d=2$, the above theorem implies that the $D^{\times}$-representation $\mathrm{Sp}(\tau)_{N,\psi}$ is isomorphic to $\bigwedge^{2}(\tau)$ if and only if $\theta(-1)^{m}=\omega_{\tau}(-1)=(-1)^{m}=(-1)^{\frac{n}{2}}$. In particular, for the quaternionic division algebra $D$ (i.e. $n=2$), the twisted Jacquet module of $\mathrm{Sp}(\tau)$ is the exterior square representation $\bigwedge^{2}(\tau)$ if and only if the central character of $\tau$ is odd.  
\end{remark}

\bibliography{tjmbib} 

\def\cprime{$'$}
\providecommand{\bysame}{\leavevmode\hbox to3em{\hrulefill}\thinspace}
\providecommand{\MR}{\relax\ifhmode\unskip\space\fi MR }
\providecommand{\MRhref}[2]{%
  \href{http://www.ams.org/mathscinet-getitem?mr=#1}{#2}
}
\providecommand{\href}[2]{#2}
\begin{thebibliography}{MgW87}

\bibitem[BZ76]{Bernsteinzelevinsky0}
I.~N. Bern{\v{s}}te{\u\i}n and A.~V. Zelevinski{\u\i}, \emph{Representations of
  the group {$GL(n,F),$} where {$F$} is a local non-{A}rchimedean field},
  Uspehi Mat. Nauk \textbf{31} (1976), no.~3(189), 5--70. \MR{0425030 (54
  \#12988)}

\bibitem[Cai23]{cai23}
Yuanqing Cai, \emph{Quaternionic {S}peh representations}, Doc. Math.
  \textbf{28} (2023), no.~4, 903--937. \MR{4705603}

\bibitem[Gar]{gar14}
Paul Garrett, \emph{Representations of {$GL_2$} and {$SL_2$} over finite
  fields}, available at
  \url{https://www-users.cse.umn.edu/~garrett/m/repns/notes_2014-15/04_finite_GL2.pdf}.

\bibitem[GMF23]{hv23}
Henniart Guy and Vignéras Marie-France, \emph{Representations of {$GL_n(D)$}
  near the identity}, 2023, available at
  \url{https://arxiv.org/abs/2305.06581}.

\bibitem[GT10]{shalika-periodsgantakeda}
Wee~Teck Gan and Shuichiro Takeda, \emph{On {S}halika periods and a theorem of
  {J}acquet-{M}artin}, Amer. J. Math. \textbf{132} (2010), no.~2, 475--528.
  \MR{2654780}

\bibitem[MgW87]{mwdegenerate}
C.~M\oe~glin and J.-L. Waldspurger, \emph{Mod\`eles de {W}hittaker
  d\'{e}g\'{e}n\'{e}r\'{e}s pour des groupes {$p$}-adiques}, Math. Z.
  \textbf{196} (1987), no.~3, 427--452. \MR{913667}

\bibitem[MS14]{ms14}
Alberto M\'{\i}nguez and Vincent S\'{e}cherre, \emph{Repr\'{e}sentations lisses
  modulo {$l$} de {${\rm GL}_m({D})$}}, Duke Math. J. \textbf{163} (2014),
  no.~4, 795--887. \MR{3178433}

\bibitem[NS24]{ns23}
Santosh Nadimpalli and Mihir Sheth, \emph{On the integrality of locally
  algebraic representations of {${\rm GL}_{2}(D)$}}, J. Number Theory
  \textbf{257} (2024), 124--145. \MR{4672189}

\bibitem[PR00]{pr00}
Dipendra Prasad and A.~Raghuram, \emph{Kirillov theory for {${\rm GL}_{2}(D)$}
  where {$D$} is a division algebra over a non-{A}rchimedean local field}, Duke
  Math. J. \textbf{104} (2000), no.~1, 19--44. \MR{1769724}

\bibitem[Pra00]{dipendragermexpansions}
Dipendra Prasad, \emph{Comparison of germ expansion on inner forms of {${\rm
  GL}(n)$}}, Manuscripta Math. \textbf{102} (2000), no.~2, 263--268.
  \MR{1771944}

\bibitem[Pra01]{pr01}
\bysame, \emph{The space of degenerate {W}hittaker models for {${\rm GL}(4)$}
  over {$p$}-adic fields}, Cohomology of arithmetic groups, {$L$}-functions and
  automorphic forms ({M}umbai, 1998/1999), Tata Inst. Fund. Res. Stud. Math.,
  vol.~15, Tata Inst. Fund. Res., Bombay, 2001, pp.~103--115. \MR{1986097}

\bibitem[SZ05]{sz05}
Allan~J. Silberger and Ernst-Wilhelm Zink, \emph{An explicit matching theorem
  for level zero discrete series of unit groups of {$p$}-adic simple algebras},
  J. Reine Angew. Math. \textbf{585} (2005), 173--235. \MR{2164626}

\bibitem[Tad90]{tad90}
Marko Tadi\'{c}, \emph{Induced representations of {${\rm GL}(n,A)$} for
  {$p$}-adic division algebras {$A$}}, J. Reine Angew. Math. \textbf{405}
  (1990), 48--77. \MR{1040995}

\bibitem[Zel80]{zelevinsky2}
A.~V. Zelevinsky, \emph{Induced representations of reductive {${\rm p}$}-adic
  groups. {II}. {O}n irreducible representations of {${\rm GL}(n)$}}, Ann. Sci.
  \'Ecole Norm. Sup. (4) \textbf{13} (1980), no.~2, 165--210. \MR{584084
  (83g:22012)}

\end{thebibliography}
\bibliographystyle{amsplain}
\end{document}